\documentclass[10pt]{article}
\usepackage{amssymb}
\usepackage{mathrsfs}
\usepackage{amsfonts}
\usepackage{tikz}

\usepackage{amsmath,amsthm,amsfonts,amssymb,amscd}

\allowdisplaybreaks[4]
\newtheorem {Problem} {Problem}[section]
\newtheorem {Theorem} [Problem]{Theorem}
\newtheorem {Lemma}[Problem]{Lemma}

\newtheorem {Corollary}[Problem]{Corollary}
\newenvironment {Proof}{\noindent {\bf Proof.}}{\hfill\ensuremath{\square}}
\newcommand*{\QEDB}{\hfill\ensuremath{\square}}

\begin{document}

\title{Spectral extremal results on  the $\alpha$-index of  graphs without  minors and star forests\thanks{This work is supported by  the National Natural Science Foundation of China (Nos.12101166, 12101165, 11971311, 12026230) and Hainan Provincial Natural Science Foundation of China (Nos. 120RC453, 120MS002), the
Montenegrin-Chinese Science and Technology Cooperation Project (No.3-12).
\newline \indent $^{\ddag}$E-Mails address:
mzchen@hainanu.edu.cn (M.-Z. Chen),\ \ \ amliu@hainanu.edu.cn (A-M. Liu),\ \ \ xiaodong@sjtu.edu.cn %(Corresponding author: X.-D.Zhang)
 }}

\author{Ming-Zhu Chen, %\thanks{E-mail: mzchen@hainanu.edu.cn},
 A-Ming Liu,%\thanks{E-mail: amliu@hainanu.edu.cn}
 \\
School of Science, Hainan University, Haikou 570228, P. R. China, \\
\and  Xiao-Dong Zhang%$^{\dagger}$
\thanks{Corresponding author (E-mail: xiaodong@sjtu.edu.cn)}
\\School of Mathematical Sciences, MOE-LSC, SHL-MAC\\
  \vspace{1cm}
Shanghai Jiao Tong University,
Shanghai 200240, P. R. China\\
Dedicated to Professor Fan Chung, with admiration and thanks}

\date{}
\maketitle

\begin{abstract}
  Let $G$ be a graph of order $n$, and let $A(G)$ and $D(G)$ be the adjacency matrix and the degree matrix of $G$ respectively. Define the convex linear combinations
  $A_\alpha (G)$ of $A (G)$ and
$D (G) $  by
$$A_\alpha (G)=\alpha D(G)+(1-\alpha)A(G)$$  for any real number $0\leq\alpha\leq1$. The  \emph{$\alpha$-index} of $G$ is the largest eigenvalue of $A_\alpha(G)$. In this paper,  we determine the  maximum $\alpha$-index and characterize all extremal graphs for  $K_r$ minor-free graphs, $K_{s,t}$ minor-free graphs, and   star-forest-free graphs  for any $0<\alpha<1$ by unified eigenvector approach,  respectively.
\end{abstract}

\section{Introduction}
 Let $G$ be an undirected simple graph with vertex set
$V(G)=\{v_1,\dots,v_n\}$ and edge set $E(G)$, where $n$ is called the order of $G$.
%Denote  $\delta(G)$ by   minimum degree of $G$.
The \emph{adjacency matrix}
$A(G)$ of $G$  is the $n\times n$ matrix $(a_{ij})$, where
$a_{ij}=1$ if $v_i$ is adjacent to $v_j$, and $0$ otherwise. The  \emph{spectral radius} of $G$, denoted by $\rho(G)$, is the largest eigenvalue of $A(G)$. The  \emph{signless Laplacian spectral radius} of $G$,  denoted by $q(G)$, is the largest eigenvalue of $Q(G)$, where $Q(G)=A(G)+D(G)$ and $D(G)$ is the degree diagonal matrix.
For $v\in V(G)$,  %the \emph{neighborhood} $N_G(v)$ of $v$  is $\{u: uv\in E(G)\}$ and
the \emph{degree} $d_G(v)$ of $v$  is the number of vertices adjacent to $v$ in $G$.
We write  $d(v)$ for $d_G(v)$  if there is no ambiguity.
 Denote by $\Delta(G)$ the maximum degree of $G$ and $\overline{G}$ the complement graph of $G$. Let $S_{n-1}$ be a star of order $n$. The \emph{center} of a star is the vertex of maximum degree in the star.  A \emph{star forest} is a forest whose components are stars.
  The \emph{centers} of a star forest are the centers of the stars in the star forest.  A graph $G$ is \emph{$H$-free} if   it  does not contain $H$ as a subgraph.  A graph $H$ is called a \emph{minor} of a graph $G$ if it can be obtained from $G$ by deleting edges, contracting edges or deleting vertices. A graph $G$ is \emph{$H$ minor-free} if   it  does not contain $H$ as a minor. % The classical \emph{Tur\'{a}n number},
%denoted by $ex(n, H)$, is the maximum number of edges in an $H$-free graph of order $n$.
 For $X,Y\subseteq V(G)$,  $e(X)$ denotes the number of edges in $G$ with two ends in $X$ and $e(X,Y)$ denotes the number of edges in $G$ with one end in $X$ and the other in $Y$.
 For two vertex disjoint graphs $G$ and $H$,  we denote by  $G\cup H$ and  $G\nabla H$  the \emph{union} of $G$ and $H$,
and the \emph{join} of $G$ and $H$ which is obtained by joining every vertex of $G$ to every vertex of $H$, respectively.
Denote by $kG$   the union of $k$ disjoint copies of $G$.
For graph notation and terminology undefined here,  readers are referred to \cite{BM}.

% Cvetkovi\'{c} called the study of the adjacency matrix
%the $A$-spectral theory, and the study of the signless Laplacian?Cthe $Q$-spectral
%theory.
To track the gradual change of $A (G)$ into $Q (G)$,
Nikiforov \cite{Nikiforov2017} proposed  and studied the convex linear combinations $A_\alpha(G)$ of $A (G)$ and
$D (G) $ defined by
$$A_\alpha (G)=\alpha D(G)+(1-\alpha)A(G)$$
for any real number $0\leq\alpha\leq1$.
  Note that $A_0(G)=A(G)$, $2A_{1/2}(G)=Q(G)$, and $A_1(G)=D(G)$. The  \emph{$\alpha$-index} of $G$ is the largest eigenvalue of $A_\alpha(G)$, denoted by
 $\rho_\alpha(G)$. Clearly, $\rho_0(G)=\rho(G)$ and $2\rho_{1/2}(G)=q(G)$.

 Let $\mathbf x=(x_u)_{u\in V(G)}$ be an eigenvector to $\rho_\alpha(G)$. By eigenequations of $A_\alpha(G)$ on any vertex $u\in V(G)$,
$$\rho_\alpha(G)x_u=\alpha d(u) x_u+(1-\alpha)\sum_{uv\in E(G)} x_v.$$
Since $A_\alpha (G)$ is a real symmetric matrix, Rayleigh's principle implies that
$$\rho_\alpha (G)=\max_{||{\bf x}||_2=1}\sum_{uv\in E(G)}(\alpha x_u^2+2(1-\alpha)x_u+\alpha x_v^2),$$
also see \cite{Nikiforov2017}.
Note that $A_\alpha(G)$ is nonnegative. By Perron-Frobenius theory of nonnegative matrices, if $G$ is connected then  $A_\alpha(G)$ has a positive eigenvector corresponding to $\rho_\alpha(G)$, also see \cite{Nikiforov2017}. In addition, if $G$ is connected and $H$ is a proper subgraph of $G$, then $$\rho_\alpha(G)>\rho_\alpha(H).$$

%Recall that the problem of maximizing the number of edges over all graphs without  fixed subgraphs is one of the cornerstones of graph theory.
%\begin{Problem}\label{P0}
%Given a graph $H$, what is the maximum number of edges of a graph $G$ of order $n$ without $H$ $?$
%\end{Problem}
%
%Many instances of Problems~\ref{P0} have been solved. For example,
% Lidick\'{y}, Liu, and Palmer \cite{Lidicky2013}  determined the Tur\'{a}n number for a forbidden forest forest  and  all  extremal graphs if the order of a  graph is  sufficiently large.
In spectral extremal graph theory, one of the central problems, which is called spectral Tur\'{a}n problem, is to find the maximum  $\rho(G)$  or $q(G)$ of a graph $G$ of order $n$ without $H$ as a subgraph or as a minor$?$ This problem is intensively investigated in the literature for
many classes of graphs. %In 2017, Nikiforov \cite{Nikiforov2017} merged Problems~\ref{P1} and \ref{P2} into one,
%namely:
%
%\begin{Problem}\label{P3}
%Given a graph $H$, what is the maximum $\rho_\alpha(G)$ of a graph $G$ of order $n$ without $H$$?$
%\end{Problem}
%of the following types:
%\begin{Problem}\label{P1}
%Given a graph  $H$, what is
%\end{Problem}
%
%
%\begin{Problem}\label{P2}
%Given a graph  $H$, what is the maximum $q(G)$ of a graph $G$ of order $n$ without $H$$?$
%\end{Problem}
 For example, Tait \cite{Tait2019} determined  the maximum spectral radius for $K_r$ minor-free graphs and $K_{s,t}$ minor free-graphs by using eigenvector. They pointed out  the extremal graphs for maximizing number of edges
and spectral radius are the same for small values of $r $ and $s $ and then differed significantly.
%\begin{Theorem}\label{Kr}
%Let $r\geq 3$ and $G$ be a $K_r$ minor-free graph of sufficiently large order $n$.
%Then
% $$\rho(G)\leq \rho(K_{r-2}\nabla K_{n-r+2})$$ with equality if and only if $G=K_{r-2}\nabla K_{n-r+2}$.
%\end{Theorem}
%
%\begin{Theorem}\label{Kst}
%Let $t\geq s\geq 2$ and $G$ be a $K_{s,t}$ minor-free graph of sufficiently large order $n$.
%Then
%\begin{eqnarray*}
%% \nonumber to remove numbering (before each equation)
%  \rho(G) &\leq& \frac{s+t-3+\sqrt{(s+t-3)^2+4((s-1)(n-s+1)-(s-2)(t-1))}}{2}
%\end{eqnarray*}
%with equality if and only if $n-s+1=pt$ and $G=K_{s-1}\nabla pK_t$.
%\end{Theorem}
Chen, Liu and Zhang \cite{Chen2019,Chen2020} determined the maximum (signless Laplacian) spectral radius for $kP_3$-free graphs. They \cite{Chen2019-2} also  determined the
maximum signless Laplacian spectral radius for $K_{2,t}$ minor-free graphs.  In addition, Nikiforov \cite{Nikiforov2011} gave an excellent survey on this topic.
%\begin{Theorem}\label{3}\cite{Chen2020}
%Let $k\geq2 $ and $G$ be a $k P_3$-free graph of order $n\geq \frac{11}{2}k^2-9k+12$. Then $$q(G)\leq q(F_{n,k})$$ with equality if and
%only if $G=F_{n,k}$.
%\end{Theorem}
%
%\begin{Theorem}\label{K2t}\cite{Chen2019-2}
%Let $t\geq4$ and $G$ be a $K_{2,t}$-minor free graph of order $n\geq t^2+4t+1$. Then $$q(G)\leq \frac{n+2t-2+\sqrt{(n-2t+2)^2+8t-8}}{2}$$ with
%equality if and only if $n-1=pt$ and $G=K_1\nabla pK_t$.
%\end{Theorem}
For more results, see \cite{He2013,Nikiforov2015,Tait2019,Tait2017,Yuan2014}.
%Let $F_{n,k}=K_{k-1}\nabla ( (p K_2)\cup K_s)$, where $n-(k-1)=2p+s$ and $0\leq s<2$.

%\begin{Theorem}\label{A}\cite{Chen2019}
%Let $k\geq2 $ and $G$ be a $k P_3$-free graph of order $n\geq 8k^2-3k$. Then $$\rho(G)\leq \rho(F_{n,k}) $$ with equality if and
%only if   $G= F_{n,k}$.
%\end{Theorem}

%Nikiforov first solved Problem~\ref{P3} when $H$ is a complete graph, which extended spectral Tur\'{a}n theorem \cite{Nikiforov2007} and signless Laplacian spectral Tur\'{a}n theorem \cite{He2013}.%In order to state these results, we need some symbols for given graphs.

%
%\begin{Theorem}\label{stt}\cite{Nikiforov2017}
%Let $r \geq 2$ and $G$ be a $K_{r+1}$-free graph of order $n$.
%
%(i) If  $0 \leq \alpha < 1-\frac{1}{r}$, then
%$$\rho_\alpha(G)\leq \rho_\alpha(T_{n,r})$$ with equality if and only if $G=T_{n,r}$.
%
%(ii) If $\alpha=1-\frac{1}{r}  $, then
%$$\rho_\alpha(G)\leq \bigg(1-\frac{1}{r}\bigg)n $$ with equality if and only if $G$ is a complete $r$-partite graph.
%
%(iii) If $1-\frac{1}{r} <\alpha <1 $, then
%$$\rho_\alpha(G)\leq \rho_\alpha(K_{r-1}\nabla \overline{K}_{n-r+1})$$  with equality if and only if $G=K_{r-1}\nabla \overline{K}_{n-r+1}$.
%\end{Theorem}

 %Surprising,  Theorem~\ref{stt} are unexpected. Nikiforov pointed out that if nothing else, Theorem~\ref{stt} shows
%that it is worth studying $A_\alpha (G)$, for it is unlikely to discover them in a different context.

Motivated by above results, we investigate the the maximum $\alpha$-index for $K_r$ minor-free graphs, $K_{s,t}$ minor-free graphs, and star-forest-free graphs.
%In this paper,  we first determine the   maximum $\alpha$-index and characterize all  extremal graphs for $K_r$-minor and $K_{s,t}$-minor free graphs for any $0<\alpha<1$, respectively.
We show the extremal graphs of $K_r$ minor-free graphs and $K_{s,t}$ minor-free graphs for maximizing
$\alpha$-index for any $0<\alpha<1$ and  sufficiently large $n$.
 Furthermore, we determine the  maximum $\alpha$-index and  characterize  all  extremal graphs for star-forest-free graphs   for any $0<\alpha<1$. %, extending the reslut of the signless Laplacian  spectral radius for $kP_3$-free graphs for large $n$ in \cite[Theorem~1.6]{Chen2010}
The main results of this paper are stated as follows.

\begin{Theorem}\label{thm2}
Let $r\geq 3$ and $G$ be a $K_r$ minor-free graph of sufficiently large order $n$.
Then for any  $0< \alpha <1$,
 $$\rho_\alpha(G)\leq \rho_\alpha(K_{r-2}\nabla \overline{K}_{n-r+2})$$ with equality if and only if $G=K_{r-2}\nabla \overline{K}_{n-r+2}$.
\end{Theorem}

\begin{Theorem}\label{thm3}
Let $t\geq s\geq 2$ and $G$ be a $K_{s,t}$ minor-free graph of sufficiently large order $n$.
Then for any  $0< \alpha <1$,
$\rho_\alpha(G)$ is no more than the largest root of $f_\alpha(x)=0$, and
equality holds if and only if $n-s+1=pt$ and $G=K_{s-1}\nabla pK_t$, where
\begin{eqnarray*}
 % \nonumber to remove numbering (before each equation)
   &&f_\alpha(x) =x^2-\Big(\alpha n+s+t-3\Big)x+\\
    && \Big(\alpha(n-s+1)+s-2\Big)\Big(\alpha(s-1)+t-1\Big)-(1-\alpha)^2(s-1)(n-s+1).
 \end{eqnarray*}
\end{Theorem}

\begin{Theorem}\label{thm1}
Let $F=\cup_{i=1}^k S_{d_i}$ be a star forest with $k\geq2$ and $d_1\geq\cdots \geq d_k\geq1$.
If $G$ is an $F$-free graph of  order $n\geq \frac{4(\sum_{i=1}^k d_i+k-2)(\sum_{i=1}^k d_i+3k-5)}{\alpha^3}$ for any $0<\alpha<1$,  then
 $\rho_\alpha(G)$ is no more than the largest root of $f_\alpha(x)=0$, and
equality holds if and only if $G=K_{k-1}\nabla H$ and $H$ is a $(d_k-1)$-regular graph of order $n-k+1$, where
\begin{eqnarray*}
 % \nonumber to remove numbering (before each equation)
   &&f_\alpha(x) =x^2-\Big(\alpha n+k+d_k-3\Big)x+\\
    && \Big(\alpha(n-k+1)+k-2)\Big)\Big(\alpha(k-1)+d_k-1\Big)-(1-\alpha)^2(k-1)(n-k+1).
 \end{eqnarray*}
 \end{Theorem}

The rest of this paper is organized as follows. In Section~2, some lemmas are presented. In Section~3, we give the proofs of Theorem~\ref{thm2} and Theorem~\ref{thm3}. In Section~4, we give the proof of Theorem~\ref{thm1} and some corollaries.

%The rest of this paper is organized as follows. In Section~2, we present some known and necessary results.
%In Section~3, we give the proofs of Theorems~\ref{thm1.1} and  \ref{thm1.2}.
%In Section~4, we present some necessary lemmas and give the proof of Theorem~\ref{thm-K33-minor free}.

\section{Preliminary}

\begin{Lemma}\label{Cor-Spec3}
Let  $0<\alpha< 1$, $k\geq2$, and  $n\geq k-1$. If $G =K_{k-1} \nabla \overline{K}_{n-k+1}$, then
% \nonumber to remove numbering (before each equation)
  $\rho_\alpha(G)\geq  \alpha(n-1)+(1-\alpha)(k-2)$.
  In particular, if $n\geq \frac{(2k-3)^2}{2\alpha^2}-\frac{8k^2-18k+9}{2\alpha}+2k(k-1)$, then
$\rho_\alpha(G)\geq \alpha n+\frac{2k-3-(2k-1)\alpha}{2\alpha}$.
\end{Lemma}

\begin{Proof}
 Set for short $\rho_\alpha =\rho_\alpha (G)$ and let ${\bf x_\alpha}=(x_v)_{v\in V(G)}$  be a positive
eigenvector to $\rho_\alpha$. By symmetry,  all vertices corresponding to $K_{k-1}$ in the representation $G := K_{k-1}  \nabla \overline{K}_{n-k+1}$ have the same eigenvector entries, denoted by $x_1$. Similarly, all  remaining vertices have the same eigenvector entries, denoted by $x_2$.
By eigenequations of $A_\alpha(G)$, we have
%\begin{eqnarray*}
%% \nonumber to remove numbering (before each equation)
% (\rho_\alpha-\alpha(n-1)) x_1 &=& (1-\alpha)(k-2)x_1+(1-\alpha)(n-k+1)x_2\\
% (\rho_\alpha-\alpha (k-1)) x_2 &= &(1-\alpha)(k-1)x_1,
%\end{eqnarray*}
%which implies that
\begin{eqnarray*}
% \nonumber to remove numbering (before each equation)
   (\rho_\alpha-\alpha(n-1)-(1-\alpha)(k-2))x_1&=& (1-\alpha)(n-k+1)x_2 \\
 (\rho_\alpha-\alpha(k-1))x_2 &=& (1-\alpha)(k-1)x_1.
\end{eqnarray*}
Then
%$$f(\rho_\alpha)\leq0,$$ where
% \begin{eqnarray*}
% % \nonumber to remove numbering (before each equation)
%   &&f_\alpha(x) =x^2-\Big(\alpha n+k+d-3\Big)x+\\
%    && \Big(\alpha(n-1)+(1-\alpha)(k-2)\Big)\Big(\alpha(k+d-2)+d-1\Big)-(1-\alpha)^2(k-1)(n-k+1).
% \end{eqnarray*}
%Hence $\rho_\alpha$ is no more than the largest root of $f_\alpha(x)=0$.
%Let $d=1$ in Lemma~\ref{spec2}. Then
$\rho_\alpha(G)$ is the largest root of $g(x)=0$, where
$$g(x)=x^2-(\alpha n+k-2)x+ (k-1)(2\alpha-1)n+(k-1)(k-k\alpha-1)=0.$$
Clearly,
\begin{eqnarray*}
         % \nonumber to remove numbering (before each equation)
           \rho_\alpha(G) &=& \frac{\alpha n+k-2+\sqrt{(\alpha n+k-2)^2-4(k-1)(2\alpha-1)n-4(k-1)(k-k\alpha-1)}}{2} \\
           &\geq& \frac{(\alpha n+k-2)+(\alpha n+k-2-(k-1)\alpha) }{2}\\
           &=&\alpha n+k-2-(k-1)\alpha\\
           &=&\alpha(n-1)+(1-\alpha)(k-2).
         \end{eqnarray*}

In addition,
\begin{eqnarray*}
% \nonumber to remove numbering (before each equation)
  &&g\bigg( \alpha n+\frac{2k-3-(2k-1)\alpha}{2\alpha}\bigg)\\
 % &=&  \bigg( \alpha n+\frac{2k-3-(2k-1)\alpha}{2\alpha} \bigg)^2-(\alpha n+k-2)\bigg( \alpha n+\frac{2k-3-(2k-1)\alpha}{2\alpha} \bigg)+\\
%  &&  (k-1)(2\alpha-1)n+(k-1)(k-k\alpha-1)\\
&=&-\frac{(1-\alpha)}{2}\bigg(n-\frac{(2k-3)^2}{2\alpha^2}+\frac{8k^2-18k+9}{2\alpha}-2k(k-1)\bigg)\\
&\leq&0,
\end{eqnarray*}
we have $$\rho_\alpha(G)\geq \alpha n+\frac{2k-3-(2k-1)\alpha}{2\alpha}.$$
\end{Proof}
\vspace{3mm}

Next we compare two lower bounds of $\rho_\alpha(G)$ in Lemma~\ref{Cor-Spec3}.

\vspace{3mm}
{\bf Remark.} Note that $$\alpha n+\frac{2k-3-(2k-1)\alpha}{2\alpha}-(\alpha(n-1)+(1-\alpha)(k-2))=\frac{((2k-2)\alpha-(2k-3))(\alpha-1)}{2\alpha}.$$
If $0<\alpha\leq \frac{2k-3}{2k-2}$, then $$\alpha n+\frac{2k-3-(2k-1)\alpha}{2\alpha}\geq\alpha(n-1)+(1-\alpha)(k-2).$$
If  $\frac{2k-3}{2k-2}<\alpha<1 $, then $$\alpha n+\frac{2k-3-(2k-1)\alpha}{2\alpha}<\alpha(n-1)+(1-\alpha)(k-2).$$

%In order to present a spectral result for a special graph, we first give the definition of  $N(\alpha)$, where $0\leq\alpha<1$:
%
%\[N(\alpha)=\left\{
%\begin{array}{cccc}
% \vspace{1mm}
%   \frac{(d-1)^2+(k-1)^2}{k-1}&& \mbox{if $\alpha=0$ }\\
%   \vspace{1mm}
% && \max\{k-1,2k-2+\frac{d-k+1}{\alpha}\} \mbox{if $0<\alpha<1$}
%\end{array}\right.
%\]

\begin{Lemma}\label{spec2}
Let $0< \alpha<1$, $d \geq 2$, $k\geq1$,  $n\geq  \max\{k-1,2k-2+\frac{d-k+1}{\alpha}\}$, and  $H$ be a graph of order $n -k+1$.  If $G= K_{k-1} \nabla H$ and
$\Delta(H) \leq d-1$, then $\rho_\alpha(G)$ is no more than the largest root of $f_\alpha(x)=0$, and
equality holds if and only if $H$ is a $(d-1)$-regular graph, where
\begin{eqnarray*}
 % \nonumber to remove numbering (before each equation)
   &&f_\alpha(x) =x^2-\Big(\alpha n+k+d-3\Big)x+\\
    && \Big(\alpha(n-k+1)+k-2\Big)\Big(\alpha(k-1)+d-1\Big)-(1-\alpha)^2(k-1)(n-k+1).
 \end{eqnarray*}
\end{Lemma}
\begin{Proof}
Let $u_1,u_2,\cdots,u_{k-1}$ be the vertices of $G$ corresponding to $K_{k-1}$ in the representation $G := K_{k-1} \nabla H$. Set for short $\rho_\alpha =\rho_\alpha (G)$ and let ${\bf x_\alpha}=(x_v)_{v\in V(G)}$  be a positive
eigenvector to $\rho_\alpha$. By symmetry,  $x_{u_1}=\cdots=x_{u_{k-1}}$.  Choose a vertex $v\in V (H)$ such that
$$x_v = \max_{z\in V (H)} x_z.$$

Since $\Delta(H)\leq d-1$ and $G=K_{n-1}\nabla H$, we have $d(v)\leq k-1+d-1=k+d-2.$ By eigenequations of $A_\alpha(G)$ on $u_1$ and $v$, we have

\begin{equation}\label{1.1}
\begin{aligned}
% \nonumber to remove numbering (before each equation)
 (\rho_\alpha-\alpha(n-1)) x_{u_1} &= (1-\alpha)(k-2)x_{u_1}+(1-\alpha)\sum_{uu_1\in E(H)}x_u\\
 &\leq(1-\alpha)(k-2)x_{u_1}+(1-\alpha)(n-k+1)x_v
     \end{aligned}
 \end{equation}
 \begin{equation}\label{1.2}
\begin{aligned}
  (\rho_\alpha-\alpha (k+d-2)) x_{v}  &\leq(\rho_\alpha-\alpha d(v)) x_{v} = (1-\alpha)(k-1)x_{u_1}+(1-\alpha)\sum_{uv\in E(H)}x_u\\
 & \leq(1-\alpha)(k-1)x_{u_1}+(1-\alpha)(d-1)x_v,
    \end{aligned}
 \end{equation}
which implies that
\begin{eqnarray*}
% \nonumber to remove numbering (before each equation)
   (\rho_\alpha-\alpha(n-1)-(1-\alpha)(k-2))x_{u_1}&\leq& (1-\alpha)(n-k+1)x_v \\
 (\rho_\alpha-\alpha(k+d-2)-(1-\alpha)(d-1))x_v &\leq& (1-\alpha)(k-1)x_{u_1}.
\end{eqnarray*}
Note that $K_{k-1} \nabla \overline{K}_{n-k+1}$ is a  subgraph of $G$. By Lemma~\ref{Cor-Spec3}, we have
$$\rho_\alpha\geq\rho_\alpha(K_{k-1} \nabla \overline{K}_{n-k+1})\geq \alpha(n-1)+(1-\alpha)(k-2)\geq\alpha(k+d-2)+(1-\alpha)(d-1).$$
Let
%\begin{equation}\label{1.3}
%f(\rho_\alpha)\leq0,
%\end{equation} where
 \begin{eqnarray*}
 % \nonumber to remove numbering (before each equation)
   &&f_\alpha(x) =x^2-\Big(\alpha n+k+d-3\Big)x+\\
    && \Big(\alpha(n-k+1)+k-2)\Big)\Big(\alpha(k-1)+d-1\Big)-(1-\alpha)^2(k-1)(n-k+1).
 \end{eqnarray*}
Then $\rho_\alpha$ is no more than the largest root of $f_\alpha(x)=0$.
 %In particular, if $\alpha=0$ then
% $$\rho_\alpha\leq  \frac{d+k-3+\sqrt{(d-k+1)^2+4(k-1)(n-k+1)}}{2}.$$
% If $\alpha=\frac{1}{2}$, then
% $$\rho_{\frac{1}{2}}\leq  \frac{d+k-3+\sqrt{(d-k+1)^2+4(k-1)(n-k+1)}}{2}.$$
If $\rho_\alpha$ is  equal to  the largest root of $f_\alpha(x)=0$, then all equalities in (\ref{1.1}) and (\ref{1.2}) hold. So  $d(v)=k+d-2$ and $x_z = x_v$ for any vertex
$z \in V (H)$. Since for any $z\in V(H)$,
\begin{eqnarray*}
% \nonumber to remove numbering (before each equation)
 (\rho_\alpha-\alpha d(z)) x_z&=& (1-\alpha)(k-1)x_{u_1}+ (1-\alpha)\sum_{uz\in E(H)}x_u\\
 & \leq& (1-\alpha)(k-1)x_{u_1}+ (1-\alpha)(d-1)x_v=(\rho_\alpha-\alpha d(v)) x_v,
\end{eqnarray*}
we have $d(z)=d(v)=d+k-2$. So $H$ is $(d-1)$-regular.
\end{Proof}
\vspace{2mm}

\section{Graphs without minors}%Proof of Theorems~\ref{thm2} and \ref{thm3}
 We first present some structural lemmas for $K_r$ minor-free graphs and $K_{s,t}$ minor-free graphs respectively.

\begin{Lemma}\label{e-minor2}\cite{Tait2019}
Let $r\geq 3$ and  $G$ be a bipartite $K_r$ minor-free graph of order $n$  with vertex partition
$K$ and $T$. Let $|K| = k$ and $|T| = n - k$. Then there is an absolute constant $C$ depending
only on $r$ such that
$$e(G) \leq Ck + (r - 2)n.$$
In particular, if $|K| = o(n)$, then $e(G) \leq (r -2 + o(1))n$.
\end{Lemma}

\begin{Lemma}\label{e-minor3}\cite{Tait2019}
 Let $G$ be a $K_r$ minor-free graph of order $n$. Assume that $(1-2\delta)n > r$, and $(1 - \delta)n > \binom{r-2}{2} +2$, and that there is a set $K$ with $|K| = r - 2$ and a set $T$ with $|T| = (1- \delta)n$ such that every vertex in $K $ is adjacent to every vertex in $T$. Then
we may add edges to $K$ to make it a clique and the resulting graph is still $K_r$ minor-free.
\end{Lemma}

\begin{Lemma}\label{e-minor5}\cite{Tait2019,Thomason2007}
Let $t\geq s \geq2$ and $G$ be a bipartite $K_{s,t}$ minor-free graph of order $n$  with vertex partition
$K$ and $T$. Let $|K| = k$ and $|T| = n - k$. Then there is an absolute constant $C$ depending
only on $s$ and $t$ such that
$$e(G) \leq Ck + (s - 1)n.$$
In particular, if $|K| = o(n)$, then $e(G) \leq (s -1 + o(1))n$.
\end{Lemma}
\begin{Lemma}\label{e-minor4}\cite{Mader1967}
 For any graph $H$, there is a constant $C$ such that if $G$ is an $H$ minor-free graph of order $n$ then
 $$e(G)\leq Cn.$$
\end{Lemma}

\vspace{2mm}
\noindent {\bf Proof of Theorem~\ref{thm2}.}
%\begin{Theorem}\label{}
%Let $r\geq 3$ and $G$ be a $K_r$ minor-free graph of sufficiently large order $n$.
%Then for any real number $0< \alpha <1$,
% $$\rho_\alpha(G)\leq \rho_\alpha(K_{r-2}\nabla K_{n-r+2})$$ with equality if and only if $G=K_{r-2}\nabla K_{n-r+2}$.
%\end{Theorem}
%\begin{Proof}
Let $G$ be a $K_r$ minor-free  graph of  sufficiently large order  $n$ with the maximum $\alpha$-index.

\vspace{2mm}
{\bf Claim~1.} $G$ is connected.

If $G$ is not connected, then we can add an edge to two components of $G$ to get a $K_r$-minor free graph with larger $\alpha$-index, a contradiction. This proves Claim~1.

\vspace{2mm}
Next let $\rho_\alpha=\rho_\alpha (G)$ and $\mathbf x=(x_v)_{v\in V(G)}$ with the maximum entry 1 be a positive eigenvector to $\rho_\alpha $. Choose an arbitrary $w\in V(G)$ with  $$x_w=\max\{x_v:v\in V(G)\}=1. $$ Set $L=\{v\in V(G): x_v> \epsilon\} $ and $S=\{v\in V(G): x_v\leq \epsilon\} $, where $\epsilon$ will be chosen later.

Since $K_{r-2}\nabla \overline{K}_{n-r+2}$ is $K_r$-minor free, by Lemma 2.1,

\begin{eqnarray}\label{I2}
% \nonumber to remove numbering (before each equation)
  \rho_\alpha &\geq& \rho_\alpha(K_{r-2}\nabla \overline{K}_{n-r+2})\geq \max\bigg\{\alpha n+\frac{2r-5-(2r-3)\alpha}{2\alpha},\alpha (n-1)\bigg\}.
\end{eqnarray}
By Lemma~\ref{e-minor4},  there is a constant $C_1$ such that
\begin{eqnarray}\label{I3}
% \nonumber to remove numbering (before each equation)
  2e(S)\leq2e(G)\leq C_1n.
\end{eqnarray}

\vspace{2mm}
%Let $C=\{v\in X:d(v)\geq \sum_{i=1}^k d_i+k-1\}$.  Since $G$ is $F$-free, $|C|\leq k-1$, otherwise we can embed an $F$ in $G$ by definition of $C$.
%Hence
%\begin{eqnarray*}
%% \nonumber to remove numbering (before each equation)
%  e(G) &=& \sum_{v\in C}d(v)+ \sum_{v\in V(G)\backslash C}d(v)\\
%   &\leq&(n-1)|C|+(n-|C|)\bigg(\sum_{i=1}^k d_i+k-2\bigg) \\
%   &=& \bigg(n- \sum_{i=1}^k d_i-k+1\bigg)|C|+\bigg(\sum_{i=1}^k d_i+k-2\bigg)n\\
%   &\leq&(k-1)\bigg(n- \sum_{i=1}^k d_i-k+1\bigg)+\bigg(\sum_{i=1}^k d_i+k-2\bigg)n\\
%   &=&\bigg(\sum_{i=1}^k d_i+2k-3\bigg)n-(k-1)\bigg(\sum_{i=1}^k d_i+k-1\bigg)\\
%   &\leq&\bigg(\sum_{i=1}^k d_i+2k-3\bigg)n
%\end{eqnarray*}
{\bf Claim~2.} There exists a constant $C_2$ such that $$|L|\leq  \frac{C_2(1-\alpha+\alpha\epsilon)}{\epsilon}.$$ In addition, $\epsilon$ can be chosen small enough that $$e(L,S)\leq (k-1+\epsilon)n.$$

By eigenequations of $A_\alpha$ on any vertex $u\in L$, we have
$$(\rho_\alpha-\alpha d(u))\epsilon<(\rho_\alpha-\alpha d(u)) x_u=(1-\alpha)\sum_{uv\in E(G)} x_v\leq (1-\alpha)d(u),$$
which implies that $$d(u)> \frac{\rho_\alpha\epsilon}{1-\alpha+\alpha\epsilon}.$$
Thus $$2e(G)=\sum_{u\in V(G)}d(u)\geq \sum_{u\in L }d(u)\geq \frac{|L|\rho_\alpha\epsilon}{1-\alpha+\alpha\epsilon},$$
which implies that
\begin{eqnarray}\label{I0}
% \nonumber to remove numbering (before each equation)
|L|\leq\frac{2e(G)(1-\alpha+\alpha\epsilon)}{\rho_\alpha\epsilon}.
\end{eqnarray}
For  sufficiently large $n$, there is a constant $C_2$ such that
$$C_2\geq \frac{2\alpha C_1}{2\alpha^2+\frac{2r-5-(2r-3)\alpha}{n}}.$$
Hence by (\ref{I2})-(\ref{I0}),
 \begin{eqnarray*}
 % \nonumber to remove numbering (before each equation)
   |L| &\leq&\frac{2e(G)}{\rho_\alpha}\cdot \frac{(1-\alpha+\alpha\epsilon)}{\epsilon}\leq \frac{C_1 n}{\alpha n+\frac{2r-5-(2r-3)\alpha}{2\alpha}}\cdot \frac{(1-\alpha+\alpha\epsilon)}{\epsilon}\\
   &=& \frac{2\alpha C_1}{2\alpha^2+\frac{2r-5-(2r-3)\alpha}{n}}\cdot\frac{1-\alpha+\alpha\epsilon}{\epsilon} \leq \frac{C_2(1-\alpha+\alpha\epsilon)}{\epsilon}.
 \end{eqnarray*}
% $$|L|\leq\frac{2e(G)(1-\alpha+\alpha\epsilon)}{\rho_\alpha\epsilon}\leq \frac{2\alpha}{2\alpha^2+\frac{2r-5-(2r-3)\alpha}{n}}\cdot\frac{1-\alpha+\alpha\epsilon}{\epsilon}\leq \frac{C_2(1-\alpha+\alpha\epsilon)}{\epsilon}.$$
Choose $\epsilon$ small enough such that $|L|\leq \epsilon n$. By Lemma~\ref{e-minor2}, $e(L,S)\leq (r-2+\epsilon)n$. This proves Claim~2.

By Claim~2, we can choose $\epsilon$ small enough such that $$2e(L)\leq C_1 |L|\leq  \frac{C_1C_2(1-\alpha+\alpha\epsilon)}{\epsilon}\leq \epsilon n.$$
%$$d(u)\geq \frac{\rho_\alpha\epsilon}{1-\alpha+\alpha\epsilon}>\bigg(\alpha n+\frac{2k-3-(2k-1)\alpha}{2\alpha}\bigg)\frac{\varepsilon}{1-\alpha+\alpha\epsilon}\geq \sum_{i=1}^kd_i+k-2,$$

%We first sum of Perron vector entry over $L$.
%\begin{eqnarray*}
%% \nonumber to remove numbering (before each equation)
%  \rho \sum_{v\in L} x_v &=&    \sum_{v\in L} \rho x_v=\sum_{v\in L} \sum_{z\in V(G)} a_{vz}x_z
%   \leq \sum_{v\in L}d_v\leq 2e(G)
%\end{eqnarray*}
%
%Hence $$\epsilon|L|\leq\sum_{v\in L} x_v\leq \frac{2e(G)}{\rho}\leq \frac{C_1n}{\sqrt{(k-1)(n-k+1)}}\leq \frac{C_1 n}{\sqrt{(k-2)n}}=\frac{C_1}{\sqrt{k-2}}\sqrt{n},$$
%Let $C_2=\frac{C_1}{\sqrt{k-2}}$. Then $|L|\leq C_2\sqrt{n}$. By Lemma~\ref{e-minor2}, $e(L,S)\leq (k-1+\epsilon)n.$

%\vspace{2mm}
%
%{\bf Claim~3.} There is a constant $C_3$ such that $2e(L)\leq C_3 \sqrt{n}$ and $2e(S)\leq C_3n$.
%
%Note that $|S|\geq n-C_1\sqrt{n}=(1-o(1))n$ for large $n$.  By Claim~1,
%$$2e(S)\leq C_1|S|\leq C_1n.$$
%By  Claims~1 and 2, if $|L|\geq \sum_{i=1}^k d_i+k$ then  $$2e(L)\leq C_1|L|\leq C_1C_2\sqrt{n}. $$
%If $|L|<\sum_{i=1}^k d_i+k$, then there is a constant $C_{31}$ such that $$e(L)=\binom{|L|}{2}\leq\binom{\sum_{i=1}^k d_i+k-1}{2}\leq C_{31}\sqrt{n}.$$ Let $C_3=\max\{C_1C_2,C_{31}\}$. Then $2e(L)\leq C_3\sqrt{n}$ and $2e(S)\leq C_3n$.

\vspace{2mm}
{\bf Claim~3.}
Let $u\in L$. Then for any $u\in L$, there is a constant $C_3$ such that
$$ d(u)\geq (1-C_3(1-x_u+\epsilon))n.$$

Since
\begin{eqnarray*}
% \nonumber to remove numbering (before each equation)
  &&\rho_\alpha \sum_{v\in V(G)}x_v\\
  &=& \sum_{v\in V(G)} \rho_\alpha x_v=\sum_{v\in V(G)}\bigg(\alpha d(v)x_v+(1-\alpha)\sum_{vz\in E(G)}x_z\bigg)\\
  &=&\alpha\sum_{v\in V(G)} d(v)x_v+(1-\alpha)\sum_{vz\in E(G)}(x_v+x_z)\\
  &=&\alpha\bigg( \sum_{v\in L} d(v)x_v+ \sum_{v\in S} d(v)x_v\bigg)+(1-\alpha)\bigg(\sum\limits_{vz\in E(L)}(x_v+x_z)+\sum\limits_{vz\in E(L,S)}(x_v+x_z)+\\
 && \sum\limits_{vz\in E(S)}(x_v+x_z) \bigg)\\
%  &=&\alpha \sum_{v\in L} d(v)x_v+\alpha \sum_{v\in S} d(v)x_v+(1-\alpha) \sum_{v\in V(G)}\sum_{vz\in E(G)}x_z\\
  &\leq&\alpha(2e(L)+e(L,S))+\alpha\epsilon (2e(S)+e(L,S))+(1-\alpha)(2e(L)+(1+\epsilon) e(L,S) + 2\epsilon e(S))\\
   &=&  2e(L)+2\epsilon e(S)+(1+\epsilon)e(L,S),
\end{eqnarray*}
we have
\begin{equation}\label{I1}
\begin{aligned}
 \sum_{v\in V(G)}x_v&\leq \frac{ 2e(L)+2\epsilon e(S)+(1+\epsilon)e(L,S)}{\rho_\alpha}\\
 &\leq \frac{ \epsilon n+\epsilon C_1n+(1+\epsilon)(r-2+\epsilon)n}{\rho_\alpha}\\
 &= \frac{ ((1+C_1)\epsilon +(1+\epsilon)(r-2+\epsilon))n}{\rho_\alpha}.
 \end{aligned}
\end{equation}
By eigenequations of $A_\alpha$ on  $u$, we have
\begin{eqnarray}\label{I4}
% \nonumber to remove numbering (before each equation)
(\rho_\alpha-\alpha d(u)) x_u&=&(1-\alpha)\sum_{uv\in E(G)} x_v\leq (1-\alpha)\sum_{v\in V(G)}x_v.
\end{eqnarray}
By  (\ref{I2}), (\ref{I1}), and (\ref{I4}), we have
\begin{eqnarray*}
% \nonumber to remove numbering (before each equation)
   d(u) &\geq& \frac{\rho_\alpha}{\alpha}-\frac{(1-\alpha)\sum_{v\in V(G)}x_v}{\alpha x_u} \\
  &\geq& \frac{\rho_\alpha}{\alpha}- \frac{(1-\alpha)( (1+C_1)\epsilon +(1+\epsilon)(r-2+\epsilon))n}{\rho_\alpha\alpha x_u}\\
   &\geq& n-1-\frac{(1-\alpha)( (1+C_1)\epsilon +(1+\epsilon)(r-2+\epsilon)}{\alpha^2 (1-\frac{1}{n})x_u}.
\end{eqnarray*}
Since $n$ is sufficiently large and $\epsilon$ is small enough, there is a constant $C_3$ such that
$$   d(u) \geq n-1-\frac{(1-\alpha)( (1+C_1)\epsilon +(1+\epsilon)(r-2+\epsilon)}{\alpha^2 (1-\frac{1}{n})x_u}\geq(1-C_3(1-x_u+\epsilon))n.$$

%Next we sum of eigenvector over all vertices in $B_u$ with aid of (1) and Lemma~\ref{spec1}. Since
%\begin{equation*}\label{I1}
%  \begin{aligned}
%% \nonumber to remove numbering (before each equation)
% \frac{1}{  \rho }|B_u|&\leq \sum_{v\in B_u}x_v=\sum_{v\in S} x_v-\sum_{\substack{v\in S\\uv\in E(G)}} x_v\\
%   &=  \sum_{v\in S}x_v-\bigg(\sum_{uv\in E(G)}x_v-\sum_{\substack{v\in L\\uv\in E(G)}}x_v\bigg)\\
%   &\leq \sum_{v\in S}x_v+  \sum_{v\in L}x_v- \sum_{uv\in E(G)}x_v\\
%   &=\sum_{v\in V(G)}x_v-\rho x_u\\
%   &\leq  \frac{ 2e(L)+2\epsilon e(S)+(1+\epsilon)e(L,S)}{\rho}-\rho x_u,
%     \end{aligned}
%\end{equation*}
% we have
%\begin{eqnarray*}
%% \nonumber to remove numbering (before each equation)
%  |B_u| &\leq&  2e(L)+2\epsilon e(S)+(1+\epsilon)e(L,S)-\rho^2x_u \\
%   &\leq&  2e(L)+2\epsilon e(S)+(1+\epsilon)e(L,S)-(k-1)(n-k+1)(1-\delta) \\
%   &\leq& C_3\sqrt{n}+\epsilon \cdot C_1n+(1+\epsilon)(k-1+\epsilon)n-(k-1)(n-k+1)(1-\delta)  \\
%  &=&  C_3\sqrt{n}+ C_3\epsilon n+(1+\epsilon)(k-1+\epsilon)n-(k-1)(1-\delta)n+ (k-1)^2(1-\delta)\\
%   &=& \Big(C_3\epsilon+(1+\epsilon)(k-1+\epsilon)-(k-1)(1-\delta)\Big)n+C_3\sqrt{n}+(k-1)^2(1-\delta)\\
% &=&  \Big(C_3\epsilon+\epsilon(k+\epsilon)+\delta(k-1)\Big)n+C_3\sqrt{n}+(k-1)^2(1-\delta)\\
%   &\leq& (C_3+k+\epsilon)(\delta +\epsilon)n+C_{41}(\delta+ \epsilon)n\\
%   &\leq& (C_3+C_{41}+k+1)(\delta+ \epsilon)n,
%\end{eqnarray*}
%where the last second inequality holds since $C_3\sqrt{n}+(k-1)^2(1-\delta)\leq C_{41}(\delta +\epsilon)n$ for some $C_{41}$ and large  $n$.
%
%Let $C_4=C_3+C_{41}+k+1$.  Then the result follows.

\vspace{2mm}

{\bf Claim~4.}
 Let $ 1 \leq s<r-2 $. Suppose that  there is a set $X$ of $s$ vertices such that $X=\{v\in V(G):x_v\geq 1-\eta ~\text{and} ~d(v)\geq(1 - \eta)n\} $, where $\eta$ is much smaller than $1$. Then there is a
constant $C_4$ and a vertex $v\in L \backslash X $ such that  $ x_{v}\geq 1 - C_4(\eta + \epsilon)$ and  $d(v)\geq (1 - C_4(\eta + \epsilon))n $.

\vspace{2mm}

By eigenequations of $A_\alpha$ on  $w$, we have
$$\rho_\alpha-\alpha d(w)=(\rho_\alpha-\alpha d(w)) x_w=(1-\alpha)\sum_{vw\in E(G)} x_v.$$
Multiplying both sides of the above inequality by $\rho_\alpha$, we have
\begin{eqnarray*}
% \nonumber to remove numbering (before each equation)
  &&\rho_\alpha(\rho_\alpha-\alpha d(w))\\
  &=& (1-\alpha)\sum_{vw\in E(G)} \rho_\alpha x_v \\
   &=& (1-\alpha)\sum_{vw\in E(G)} \bigg(\alpha d(v)x_v+(1-\alpha)\sum_{uv\in E(G)} x_u\bigg)\\
 &=& (1-\alpha)\sum_{vw\in E(G)} \alpha d(v)x_v+(1-\alpha)^2\sum_{vw\in E(G)}\sum_{uv\in E(G)} x_u\\
 &\leq&(1-\alpha)\bigg(\sum_{v\in V(G)} \alpha d(v)x_v-\alpha d(w)\bigg)+(1-\alpha)^2\sum_{uv\in E(G)}( x_u+x_v)-\\
 &&(1-\alpha)^2\sum_{vw\in E(G)}x_v\\
&=&\alpha(1-\alpha)\sum_{uv\in E(G)}( x_u+x_v)-\alpha(1-\alpha)d(w)+(1-\alpha)^2\sum_{uv\in E(G)}( x_u+x_v)-\\
&&(1-\alpha)(\rho_\alpha-\alpha d(w))\\
&=&(1-\alpha)\sum_{uv\in E(G)}( x_u+x_v)-(1-\alpha)\rho_\alpha,
\end{eqnarray*}
which implies that $$\sum_{uv\in E(G)}( x_u+x_v)\geq \frac{\rho_\alpha(\rho_\alpha+1-\alpha-\alpha d(w))}{1-\alpha}.$$
On the other hand,
\begin{eqnarray*}
% \nonumber to remove numbering (before each equation)
  \sum\limits_{uv\in E(G)}(x_u+x_v)&=&\sum\limits_{uv\in E(L,S)}(x_u+x_v)+\sum\limits_{uv\in E(S)}(x_u+x_v) +\sum\limits_{uv\in E(L)}(x_u+x_v)\\
   &\leq& \sum\limits_{uv\in E(L,S)}(x_u+x_v)+ 2\epsilon e(S)+2e(L)\\
   &\leq&\epsilon e(L,S)+\sum_{\substack{uv\in E( L\backslash X,S)\\u\in L\backslash X}}x_u+\sum_{\substack{uv\in E(L\cap X,S)\\u\in L\cap X }}x_u+2\epsilon e(S)+2e(L).
\end{eqnarray*}
Let  $ t=|L\cap X|$. Combining with (\ref{I2}), we have
\begin{eqnarray*}
% \nonumber to remove numbering (before each equation)
 &&\sum_{\substack{uv\in E( L\backslash X,S)\\u\in L\backslash X}}x_u\\
 &\geq&  \frac{\rho_\alpha(\rho_\alpha+1-\alpha-\alpha d(w))}{1-\alpha}- 2\epsilon e(S)-2e(L)-\epsilon e(L,S)-\sum_{\substack{uv\in E(L\cap X,S)\\u\in L\cap X }}x_u\\
    % \nonumber to remove numbering (before each equation)
       &\geq& \bigg(\frac{\alpha n}{1-\alpha}+\frac{2r-5-(2r-3)\alpha}{2\alpha(1-\alpha)}\bigg)\bigg(\alpha n+\frac{2r-5-(2r-3)\alpha}{2\alpha}+1-\alpha-\alpha n+\alpha\bigg)-\\
       &&\epsilon C_1n-\epsilon n-\epsilon(r-2+\epsilon) n-tn\\
       &=&\bigg(\frac{\alpha n}{1-\alpha}+\frac{2r-5-(2r-3)\alpha}{2\alpha(1-\alpha)}\bigg)\frac{(2r-5)(1-\alpha)}{2\alpha}-(\epsilon C_1+\epsilon +\epsilon(r-2+\epsilon) +t)n\\
       &=&\bigg(r-\frac{5}{2}-t-\epsilon(C_1+\epsilon+r-1)\bigg)n+\frac{(2r-5)^2-(2r-3)(2r-5)\alpha}{4\alpha^2}\\
       &\geq&\bigg(r-\frac{5}{2}-t-\epsilon(C_1+\epsilon+r)\bigg)n
    \end{eqnarray*}
 In addition,% $$|Y|=e(L,S)-e(X, S)\leq (k-1)(n-k+1)-s(1-\eta)n+\binom{k-2}{2}.$$
\begin{eqnarray*}
% \nonumber to remove numbering (before each equation)
  e(L\backslash X,S)&=& e(L,S)-e(L\cap X, S) \\
   &\leq & (r-2+\epsilon)n-t(1-\eta)n+t(t-1)+t(|L|-t) \\
   &\leq&(r-2+\epsilon)n-t(1-\eta)n+t(t-1)+t(\epsilon n-t)\\
 &\leq&(r-2+2\epsilon-t(1-\eta-\epsilon))n
\end{eqnarray*}
Note that for any $\eta>0$, there exists a constant $C_{4}'$ such that $C_4'\eta\geq \frac{1}{2}$.
%Let $$f(s)=\frac{r-\frac{5}{2}-s-\epsilon(C_1+\epsilon+r)}{r-2+2\epsilon-s(1-\eta)}.$$  It is easy to see that $f(s)$ is decreasing with respect to $1\leq s\leq k-2$.
%and $N=\{vw\in E(L,S): \{v,w\}\cap K\neq \emptyset\}$.  By  definition and Claim~4, $$|N|\geq s(1 - \eta)n$$ and
%$$|M|=e(L,S)-|N| \leq (k-1+\epsilon)n-s(1 - \eta)n=(k-1+\epsilon-s+s \eta)n.$$
%The definition of $K $ and  Lemma~\ref{e-minor2} give that the number of edges with one endpoint
%in $S $ and one endpoint in $L$ which is not in $K$ is at most $(r - 2 + o(1))n - k(1 -\eta)n$.
Then there is a vertex $v\in L\backslash X$ such that \begin{eqnarray*}
% \nonumber to remove numbering (before each equation)
 x_v&\geq &\frac{\sum_{\substack{uv\in E( L\backslash X,S)\\u\in L\backslash X}}x_u}{e(L\backslash X,S)} \\
 &\geq& \frac{(r-\frac{5}{2}-t-\epsilon(C_1+\epsilon+r))n}{(r-2+\epsilon-t(1-\eta-\epsilon))n}\\
     &=&\frac{r-\frac{5}{2}-t-\epsilon(C_1+\epsilon+r)}{r-2+\epsilon-t(1-\eta-\epsilon)}\\
     &=&1-\frac{\frac{1}{2}+t\eta+\epsilon(C_1+\epsilon+t+1)}{r-2+\epsilon-t(1-\eta-\epsilon)}\\
     &\geq&1-\frac{\frac{1}{2}+(r-3)\eta+\epsilon(C_1+\epsilon+r-2)}{r-2+\epsilon-(r-3)(1-\eta-\epsilon)}\\
      &=&1-\frac{\frac{1}{2}+(r-3)\eta+\epsilon(C_1+\epsilon+r-2)}{1+(r-2)\epsilon+(r-3)\eta}\\
      &\geq &1-\frac{\max\{C_1+\epsilon+r-2,C_4'+r-3\}}{1+(r-2)\epsilon+(r-3)\eta}(\eta+\epsilon)
\end{eqnarray*}
By Claim~3, Claim~4 follows directly.

\vspace{2mm}
If we start with $w$ and iteratively apply Claim~4,   then for any $\delta> 0$,
we can choose $\epsilon$ small enough that $G$ contains a set $X$ with  $r-2$ vertices such that their common
neighborhood of size is at least $(1 - \delta)n$ and  each eigenvector entry is at least $1 - \delta$.
From now on, denote by $K$  the set $X$ with  $r-2$ vertices  mentioned above. Let $T$ be the common neighborhood of $K$ and $R=V(G)\backslash (K\cup T)$.
Clearly,  $|K|=r-2$, $|T|\geq (1 - \delta)n$, and $|R|\leq \delta n$.

\vspace{2mm}
{\bf Claim~5.}  $K$ induces a clique and $T$ is an independent set.

If $K$ does not induce a clique,  then we can add all possible edges to make it a clique. By Lemma~\ref{e-minor3}, the resulting graph $G'$ is still $K_r$ minor-free.
Since $G$ is connected,  $\rho_\alpha(G')>\rho_\alpha(G)$, a contradiction. Hence $K$ induces a clique.
If there is an edge in $T$, then there is a $K_r$ minor in $G$,  a contradiction. Thus $T$ is an independent set. This proves Claim~5.

\vspace{2mm}
{\bf Claim~6.} For any $v\in V(G)\backslash K$, we have $x_v\leq \frac{\sqrt{\alpha}}{C_1}$, where $C_1$ is the constant in  (\ref{I3}).

Since  $G$ is  $K_r$-minor free, any vertex in $R$ can be adjacent to at most one vertex in $T$. By the definition of $R$, every vertex in $R$ can be adjacent to at most $r-3$ vertices in $K$.  In addition, by Claim~5, $T$ is an independent set and thus any vertex in $T$ has at most $r-2+|R|$ neighbors. Hence for any vertex $v\in V(G)\backslash K$,
\begin{eqnarray}\label{1}
% \nonumber to remove numbering (before each equation)
  d(v)&\leq& r-2+|R|\leq r-2+\delta n.
\end{eqnarray}
  Since $R$ is also $K_r$-minor free, we have
  $$2e(R)\leq C_1|R|\leq C_1\delta n.$$
By eigenequations of $A_\alpha(G)$, we have
\begin{eqnarray*}
% \nonumber to remove numbering (before each equation)
\alpha(n-1)  \sum_{u\in  R}x_u&\leq&\rho_\alpha\sum_{u\in  R}x_u = \sum_{u\in  R}\bigg(\alpha d(u)x_u+(1-\alpha) \sum_{uv\in E(G)}x_v\bigg)\\
  &\leq& \sum_{u\in  R}( \alpha d(u)+(1-\alpha)d(u))
   = \sum_{u\in  R} d(u)\\
   &\leq&2e(R)+(r-2)|R|\leq C_1\delta n+(r-2)\delta n\\&=&(C_1+r-2)\delta n,
\end{eqnarray*}
which implies that
\begin{equation}\label{2}
\sum_{u\in  R}x_u \leq  \frac{(C_1+r-2)\delta n }{ \alpha(n-1)}.
\end{equation}
By eigenequations of $A_\alpha(G)$ on  any vertex $v\in V(G)\backslash K$, we have
\begin{eqnarray}\label{3}
% \nonumber to remove numbering (before each equation)
 (\rho_\alpha -\alpha d(v))x_v &=& (1-\alpha) \sum_{uv\in E(G)}x_u
  \leq (1-\alpha)\bigg(r-2+ \sum_{u\in  R}x_u\bigg).
\end{eqnarray}
By  Lemma~\ref{e-minor3} and  (\ref{1})-(\ref{3}), we have
\begin{eqnarray*}
% \nonumber to remove numbering (before each equation)
  x_v&\leq&\frac{(1-\alpha)(r-2+ \sum_{u\in  R}x_u )}{\rho_\alpha -\alpha  d(v)}\\
  &\leq&\frac{(1-\alpha)\bigg(r-2+ \frac{(C_1+r-2)\delta n }{\alpha(n-1)}\bigg)}{\alpha(n-1) -\alpha ( r-2+\delta n)} \\
   &=&  \frac{(1-\alpha)\bigg(r-2+ \frac{(C_1+r-2)\delta  }{\alpha(1-\frac{1}{n})}\bigg)}{\alpha ( (1-\delta) n-r+1)}\\
\end{eqnarray*}
Then we can choose $\epsilon$ small enough to make $\delta$ small enough to get the result.

\vspace{2mm}
{\bf Claim~7.}  $R$ is empty.

If $R$ is not empty, then there exists a vertex $v\in R$ such that $v$ has at most $C_1$ neighbors in $R$.
Let $H$ be a graph obtained from $G$ by removing all edges incident with $v$ and then connecting $v$ to each vertex in $K$.
Since $K$ induces a clique, $H$ is still $K_r$ minor-free. Let $u\in K$ be the vertex not adjacent to $v$. Then
\begin{eqnarray*}
% \nonumber to remove numbering (before each equation)
&&  \rho_\alpha(H)- \rho_\alpha\\
   &\geq& \frac{{\bf x}^T A_\alpha(H){\bf x}}{ {\bf x}^T{\bf x}}-\frac{{\bf x}^T A_\alpha{\bf x}}{ {\bf x}^T{\bf x}}\\
   &\geq& \frac{1}{{\bf x}^T{\bf x}}\bigg(\alpha x_u^2+2(1-\alpha)x_ux_v+\alpha x_v^2-\sum_{vz\in E(G), z\notin K} (\alpha x_z^2+2(1-\alpha)x_vx_z+\alpha x_z^2)\bigg)\\
   &\geq&\frac{1}{{\bf x}^T{\bf x}}\bigg(2\alpha (1-\delta)^2)-\frac{\alpha(C_1+1)(\alpha+2(1-\alpha)+\alpha)}{C_1^2} \bigg)\\
   &=&  \frac{2\alpha}{{\bf x}^T{\bf x}}\bigg( (1-\delta)^2-\frac{C_1+1}{C_1^2}\bigg)
\end{eqnarray*}
Choose $\epsilon$ small enough so that $(1-\delta)^2> \frac{C_1+1}{C_1^2}$. Then
$\rho_\alpha(H)>\rho_\alpha$, a contradiction. This proves Claim~7.
%\end{Proof}

By Claims~6 and 7, $G=K_{r-2}\nabla \overline{K}_{n-r+2}$. This completes the proof.
\QEDB
\vspace{3mm}

%\section{Proof of Theorem~\ref{}}
\noindent {\bf Proof of Theorem~\ref{thm3}.}
Let $G$ be a $K_{s,t}$ minor-free  graph of   order  $n$ with the maximum $\alpha$-index.

Similarly to the proof of Claim~1 in Theoren~\ref{thm2}, $G$ is connected.
Next let $\rho_\alpha=\rho_\alpha (G)$ and $\mathbf x=(x_v)_{v\in V(G)}$ be a positive eigenvector to $\rho_\alpha $ such that  $w\in V(G)$ and $$x_w=\max\{x_u:u\in V(G)\}=1.$$

Set $L=\{v\in V(G): x_v> \epsilon\} $ and $S=\{v\in V(G): x_v\leq \epsilon\} $, where $\epsilon$ will be chosen later.

\vspace{2mm}
{\bf Claim~1.}  For any $\delta> 0$,  if we choose $\epsilon$  small enough, then
 $G$ contains a set $K$ with  $s-1$ vertices such that their common
neighborhood of size is at least $(1 - \delta)n$ and  each eigenvector entry is at least $1 - \delta$.

We omit the proof of  Claim~1 as it is similar to the proofs of  of Claims~2-4  in Theorem~\ref{thm2}.

\vspace{2mm}

Let $T$ be the common neighborhood of $K$ and $R=V(G)\backslash (K\cup T)$.
%Clearly,  $|K|=s-1$, $|T|\geq (1 - \delta)n$, and $|R|\leq \delta n$.

{\bf Claim~2.}  $R$ is empty.

Noting  any vertex in $R\cup T$ has at most $t-1$  neighbors in $R\cup T$ as $G$ is $K_{s,t}$ minor-free. In addition,  noting the graph obtained from $G$ by adding a vertex adjacent to every vertex in $K$ is still $K_{s,t}$ minor-free.  The proof of Claim~2 is similar to the proofs of Claims~6 and 7. Hence it is omitted here.

\vspace{2mm}
Now $|K|=s-1$ and $|T|=n-s+1$. Let $H$ be the subgraph of $G$ induced by $T$. Now $G=G[K]\nabla H$. Since $G$ is $K_{s,t}$ minor-free, $\Delta(H)\leq t-1$. %Let $G'$ be the graph obtained from $G$ by  adding edges to $K$ to make it a clique. Then $$\rho_\alpha(G)\leq \rho_\alpha (G')$$
%with equality  if and only if $K$ induces a clique in $G$. By Lemma~\ref{spec2}, $\rho_\alpha(G')$ is no more than the largest root of $f_\alpha(x)=0$, where $$x^2-(\alpha n+s+t-3)x+(\alpha(n-s+1)+s-2))(\alpha(s-1)+t-1)-(1-\alpha)^2(s-1)(n-s+1)=0$$ and equlity holds if and  only if $G'=K_{s-1}\nabla H$, where $H$ is a $(t-1)$-regular garph. %Let $\mathcal{G}_{s,t}=\{G: G=K_{s-1}\nabla H, \mbox{where $H$ is a $(t-1)$-regular graph of order $n-s+1$}\}$.
%Hence $\rho_\alpha (G)$ is no more than the largest root of $f_\alpha(x)=0$. If  euquality holds, then $G=K_{s-1}\nabla H$, where $H$ is a $(t-1)$-regular garph.
%Hence we just need to prove that equality can hold if and only if $G=K_{s-1}\nabla pK_t$, where $n-s+1=pt$.

First suppose that $K$ induces a clique. By Lemma~\ref{spec2}, $\rho_\alpha(G)$ is no more than the largest root of $f_\alpha(x)=0$, where $$x^2-(\alpha n+s+t-3)x+(\alpha(n-s+1)+s-2))(\alpha(s-1)+t-1)-(1-\alpha)^2(s-1)(n-s+1)=0$$ and equality holds if and  only if $G=K_{s-1}\nabla H$, where $H$ is a $(t-1)$-regular graph.
 It suffices to  prove that equality can hold if and only if $G=K_{s-1}\nabla pK_t$, where $n-s+1=pt$.
 Suppose that $H$ has a connected component $H_1$ that is not isomorphic to $K_t$ and set $h:=|V(H_1)|$. Clearly $H_1$ is a $(t-1)$-regular graph of order $h\geq t+1$. If $h=t+1$, then any two nonadjacent vertices in $H$ have $t-1$ common neighbors, which combining with clique $K_{s-1}$  yields  $K_{s,t}$,  a contradiction. Thus $h\geq t+2$. Note that $G$ is $K_{s,t}$ minor-free, we have $H_1$ is $K_{1,t}$ minor-free.
 Hence $$e(H_1)\leq h+\frac{t(t-3)}{2},$$ see \cite{Ding2001}. However, since $H_1$ is  a $(t-1)$-regular graph of order $h$, we have $$e(H_1)=\frac{h(t-1)}{2}>h+\frac{t(t-3)}{2},$$ a contradiction. Hence $H$ is the union of disjoint complete graphs of order $t$, i.e.,  $G=K_{s-1}\nabla pK_t$, where $n-s+1=pt$.

Next suppose that $K$ does not induce a clique. Let $G'$ be the graph obtained from $G$ by  adding edges to $K$ to make it a clique.  Then $\rho_\alpha(G)<\rho_\alpha (G')$. By Lemma~\ref{spec2}, $\rho_\alpha(G')$ is no more than the largest root of $f_\alpha(x)=0$, and thus $\rho_\alpha(G)$ is less than the largest root of $f_\alpha(x)=0$.
 This  completes the proof.\QEDB

\vspace{3mm}
Let $\alpha=\frac{1}{2}$. It is  easy to get the following corollary for $q(G)$.
\begin{Corollary}\label{Cor4}
Let $t\geq s\geq 2$ and $G$ be a $K_{s,t}$ minor-free graph of sufficiently large order $n$.
Then
 $$q(G)\leq \frac{n+2s+2t-6+\sqrt{(n+2s-2t-2)^2+8(s-1)(t-s+1)}}{2}$$ with equality if and only if $n-s+1=pt$ and $G=K_{s-1}\nabla pK_t$.
\end{Corollary}

\section{Graphs without star forests}%Proof of Theorems~\ref{thm2}
In this section, we present the proof of Theorem~\ref{thm1} and some corollaries.  \begin{Lemma}\label{e1}
Let $F=\cup_{i=1}^k S_{d_i}$ be a star forest with $k\geq2$ and $d_1\geq\cdots \geq d_k\geq1$.
If $G$ is an $F$-free graph of order $n\geq  \sum_{i=1}^k d_i+k$, then
$$e(G)\leq\frac{1}{2}\bigg(\sum_{i=1}^k d_i+2k-3\bigg)n-\frac{1}{2}(k-1)\bigg(\sum_{i=1}^k d_i+k-1\bigg).$$
\end{Lemma}

\begin{Proof}
Let $C=\{v\in V(G):d(v)\geq \sum_{i=1}^k d_i+k-1\}$.  Since $G$ is $F$-free, $|C|\leq k-1$, otherwise we can embed an $F$ in $G$ by the definition of $C$.
Hence
\begin{eqnarray*}
% \nonumber to remove numbering (before each equation)
  2e(G) &=& \sum_{v\in C}d(v)+ \sum_{v\in V(G)\backslash C}d(v)\\
   &\leq&(n-1)|C|+(n-|C|)\bigg(\sum_{i=1}^k d_i+k-2\bigg) \\
   &=& \bigg(n- \sum_{i=1}^k d_i-k+1\bigg)|C|+\bigg(\sum_{i=1}^k d_i+k-2\bigg)n\\
   &\leq&(k-1)\bigg(n- \sum_{i=1}^k d_i-k+1\bigg)+\bigg(\sum_{i=1}^k d_i+k-2\bigg)n\\
   &=&\bigg(\sum_{i=1}^k d_i+2k-3\bigg)n-(k-1)\bigg(\sum_{i=1}^k d_i+k-1\bigg)\\
\end{eqnarray*}
This completes the proof.
\end{Proof}

Next we prove the following result for star-forest-free connected graphs, which plays an important role in the proof of Theorem~\ref{thm1}.
\begin{Theorem}\label{c-thm2}
Let $F=\cup_{i=1}^k S_{d_i}$ be a star forest with $k\geq2$ and $d_1\geq\cdots \geq d_k\geq2$.
  If $G$ is an $F$-free connected graph of   order $n\geq \frac{4(\sum_{i=1}^k d_i+k-2)(\sum_{i=1}^k d_i+3k-5)}{\alpha^2}$ for any $0<\alpha<1$,
 then $\rho_\alpha(G)$ is no more than the largest root of $f_\alpha(x)=0$ and
equality holds if and only if $G=K_{k-1}\nabla H$ and $H$ is a $(d_k-1)$-regular graph of order $n-k+1$, where
\begin{eqnarray*}
 % \nonumber to remove numbering (before each equation)
   &&f_\alpha(x) =x^2-\Big(\alpha n+k+d_k-3\Big)x+\\
    && \Big(\alpha(n-k+1)+k-2\Big)\Big(\alpha(k-1)+d_k-1\Big)-(1-\alpha)^2(k-1)(n-k+1).
 \end{eqnarray*}
\end{Theorem}
\begin{Proof}
Let $G$ be an $F$-free connected  graph  of  order  $n$ with the maximum $\alpha$-index.
Set for short $A_\alpha=A_\alpha(G)$ and $\rho_\alpha=\rho_\alpha(G)$. Let $\mathbf x_\alpha=(x_v)_{v\in V(G)}$ be a positive eigenvector to  $\rho_\alpha $ such that $w\in V(G)$ and $$x_w=\max\{x_u:u\in V\}=1.$$

 Since $K_{k-1}\nabla \overline{K}_{n-k+1}$ is $F$-free, it follows from Lemma~\ref{Cor-Spec3} that
$$\rho_\alpha\geq \rho_\alpha(K_{k-1}\nabla \overline{K}_{n-k+1})\geq \alpha n+\frac{2k-3-(2k-1)\alpha}{2\alpha}. $$
Let $L=\{u\in V(G): x_u> \epsilon\} $ and $S=\{u\in V(G): x_u\leq \epsilon\} $, where %$$\frac{(1-\alpha)\big(\sum_{i=1}^k d_i+k-2\big)}{\alpha\big(n-\sum_{i=1}^kd_i-k+2\big)}\leq\epsilon < \frac{1}{\sum_{i=1}^k 4d_i+10k-16}.$$
$$\epsilon = \frac{1}{4(\sum_{i=1}^k d_i+3k-5)}.$$

{\bf Claim.} $|L|=k-1$.

First suppose that $|L|\geq k$.
By eigenequations of $A_\alpha$ on any vertex $u\in L$, we have
$$(\rho_\alpha-\alpha d(u))\epsilon<(\rho_\alpha-\alpha d(u)) x_u=(1-\alpha)\sum_{uv\in E(G)} x_v\leq (1-\alpha)d(u),$$
which implies that $$d(u)> \frac{\rho_\alpha\epsilon}{1-\alpha+\alpha\epsilon}\geq\bigg(\alpha n+\frac{2k-3-(2k-1)\alpha}{2\alpha}\bigg)\frac{\epsilon}{1-\alpha+\alpha\epsilon}\geq \sum_{i=1}^kd_i+k-2,$$
where the last inequality holds as $\epsilon\geq \frac{2\alpha(1-\alpha)(\sum_{i=1}^k d_i+k-2)}{2\alpha^2(n-\sum_{i=1}^k d_i-k+2)-(2k-1)\alpha+2k-3}$.
Thus $$d(u)\geq\sum_{i=1}^kd_i+k-1.$$
Then we can embed an $F$ in $G$ with all centers in $L$, a contradiction.

Next suppose that $|L|\leq k-2$. Then $$e(L)\leq \binom{|L|}{2}\leq \frac{1}{2}(k-2)(k-3) $$ and $$e(L,S)\leq (k-2)(n-k+2).$$ In addition,
by Lemma~\ref{e1}, $$e(S)\leq e(G)\leq\frac{1}{2}\bigg(\sum_{i=1}^k d_i+2k-3\bigg)n.$$
By eigenequations of $A_\alpha$ on  $w$, we have
$$\rho_\alpha-\alpha d(w)=(\rho_\alpha-\alpha d(w)) x_w=(1-\alpha)\sum_{vw\in E(G)} x_v.$$
Multiplying both sides of the above equality by $\rho_\alpha$, we have
\begin{eqnarray*}
% \nonumber to remove numbering (before each equation)
  &&\rho_\alpha(\rho_\alpha-\alpha d(w))\\
  &=& (1-\alpha)\sum_{vw\in E(G)} \rho_\alpha x_v \\
   &=& (1-\alpha)\sum_{vw\in E(G)} \bigg(\alpha d(v)x_v+(1-\alpha)\sum_{uv\in E(G)} x_u\bigg)\\
 &=& (1-\alpha)\sum_{vw\in E(G)} \alpha d(v)x_v+(1-\alpha)^2\sum_{vw\in E(G)}\sum_{uv\in E(G)} x_u\\
 &\leq&(1-\alpha)\bigg(\sum_{v\in V(G)} \alpha d(v)x_v-\alpha d(w)\bigg)+(1-\alpha)^2\sum_{uv\in E(G)}( x_u+x_v)-(1-\alpha)^2\sum_{vw\in E(G)}x_v\\
&=&\alpha(1-\alpha)\sum_{uv\in E(G)}( x_u+x_v)-\alpha(1-\alpha)d(w)+(1-\alpha)^2\sum_{uv\in E(G)}( x_u+x_v)-\\
&&(1-\alpha)(\rho_\alpha-\alpha d(w))\\
&=&(1-\alpha)\sum_{uv\in E(G)}( x_u+x_v)-(1-\alpha)\rho_\alpha,
\end{eqnarray*}
which implies that $$\sum_{uv\in E(G)}( x_u+x_v)\geq \frac{\rho_\alpha(\rho_\alpha+1-\alpha-\alpha d(w))}{1-\alpha}.$$
On the other hand,
\begin{eqnarray*}
% \nonumber to remove numbering (before each equation)
  \sum\limits_{uv\in E(G)}(x_u+x_v)&=&\sum\limits_{uv\in E(L,S)}(x_u+x_v)+\sum\limits_{uv\in E(S)}(x_u+x_v) +\sum\limits_{uv\in E(L)}(x_u+x_v)\\
   &\leq& \sum\limits_{uv\in E(L,S)}(x_u+x_v)+ 2\epsilon e(S)+2e(L)\\
   &\leq& \sum\limits_{uv\in E(L,S)}(x_u+x_v)+\epsilon\bigg(\sum_{i=1}^k d_i+2k-3\bigg)n+(k-2)(k-3).
\end{eqnarray*}
Hence
\begin{eqnarray*}
% \nonumber to remove numbering (before each equation)
  &&\sum\limits_{uv\in E(L,S)}(x_u+x_v)\\
  &\geq& \frac{\rho_\alpha(\rho_\alpha+1-\alpha-\alpha d(w))}{1-\alpha}-\epsilon\bigg(\sum_{i=1}^k d_i+2k-3\bigg)n-(k-2)(k-3)\\\
 &\geq&\bigg(\frac{\alpha n}{1-\alpha}+\frac{2k-3-(2k-1)\alpha}{2\alpha(1-\alpha)}\bigg)\bigg(\alpha n+\frac{2k-3-(2k-1)\alpha}{2\alpha}+1-\alpha-\alpha (n-1)\bigg)-\\
 &&\epsilon\bigg(\sum_{i=1}^k d_i+2k-3\bigg)n-(k-2)(k-3)\\
%  &\geq&\bigg(\frac{\alpha n}{1-\alpha}+\frac{2k-3-(2k-1)\alpha}{2\alpha(1-\alpha)}\bigg)\frac{(2k-3)(1-\alpha)}{2\alpha}-2\epsilon\bigg(\sum_{i=1}^k d_i+2k-3\bigg)n-\\
%  &&(k-2)(k-3)\\
  &=&\bigg(k-\frac{3}{2}-\epsilon\bigg(\sum_{i=1}^k d_i+2k-3\bigg)\bigg)n+\frac{(2k-3)^2-(2k-3)(2k-1)\alpha}{4\alpha^2}-(k-2)(k-3),
\end{eqnarray*}
where the second inequality holds as $d(w)\leq n-1$.
On the other hand, by the definition of $L$ and $S$, we have
$$\sum\limits_{uv\in E(L,S)}(x_u+x_v)\leq(1+\epsilon)e(L,S)\leq (1+\epsilon)(k-2)(n-k+2).$$
Thus
\begin{eqnarray*}
% \nonumber to remove numbering (before each equation)
&&(1+\epsilon)(k-2)(n-k+2)\\
 &\geq& \bigg(k-\frac{3}{2}-\epsilon\bigg(\sum_{i=1}^k d_i+2k-3\bigg)\bigg)n+\frac{(2k-3)^2-(2k-3)(2k-1)\alpha}{4\alpha^2}-\\
 &&(k-2)(k-3),
\end{eqnarray*}
which implies that
$$\bigg(-\frac{1}{2}+\epsilon\bigg(\sum_{i=1}^k d_i+3k-5\bigg)\bigg)n\geq\frac{(2k-3)^2-(2k-3)(2k-1)\alpha}{4\alpha^2}+(k-2)(1+\epsilon(k-2)).$$
Since $\epsilon = \frac{1}{4(\sum_{i=1}^k d_i+3k-5)}$, we have
\begin{eqnarray*}
% \nonumber to remove numbering (before each equation)
  n &\leq &- \frac{(2k-3)^2-(2k-3)(2k-1)\alpha}{\alpha^2}- 4(k-2)\bigg(1+\frac{k-2}{4(\sum_{i=1}^k d_i+3k-5)}\bigg)\\
  &<& \frac{(2k-3)(2k-1)}{\alpha}\\
  &<&\frac{4(\sum_{i=1}^k d_i+k-2)(\sum_{i=1}^k d_i+3k-5)}{\alpha^2},
\end{eqnarray*}
a contradiction.
This proves the Claim.

By Claim, $|L|=k-1$ and thus $|S|=n-k+1$.  Then  the subgraph $H$ induced by $S$ in $G$ is $S_{d_k}$-free. Otherwise, we can embed an $F$ in $G$ with $k-1$ centers in $L$ and a center in $S$ as $d(u)\geq\sum_{i=1}^kd_i+k-1$ for any $u\in L$, a contradiction.  Now $\Delta(H)\leq d_k-1$. Note that the resulting graph obtained  from $G$ by adding all edges in $L$ and all edges with one end in $L$ and the other in $S$ is also $F$-free and its spectral radius  increase strictly. By the extremality of $G$, we have $G=K_{k-1}\nabla H$. By Lemma~\ref{spec2} and  the extremality of $G$, $\rho_\alpha$ is no more than largest root of $f_{\alpha}(x)=0$,
and  $\rho_\alpha$ is equal to the largest root of $f_{\alpha}(x)=0$ if and only if $H$ is a $(d_k-1)$-regular graph of order $n-k+1$,  where
\begin{eqnarray*}
 % \nonumber to remove numbering (before each equation)
   &&f_\alpha(x) =x^2-\Big(\alpha n+k+d_k-3\Big)x+\\
    && \Big(\alpha(n-k+1)+k-2\Big)\Big(\alpha(k-1)+d_k-1\Big)-(1-\alpha)^2(k-1)(n-k+1).
 \end{eqnarray*}
%$\rho_\alpha(G)$ is no more than the largest root of $f_\alpha(x)=0$ and
%equality holds if and only if $H$ is a $(d_k-1)$-regular graph, where
%\begin{eqnarray*}
% % \nonumber to remove numbering (before each equation)
%   &&f_\alpha(x) =x^2-\Big(\alpha(n+d_k-1)+k+d-3\Big)x+\\
%    && \Big(\alpha(n-1)+(1-\alpha)(k-2)\Big)\Big(\alpha(k+d_k-2)+d_k-1\Big)-(1-\alpha)^2(k-1)(n-k+1).
% \end{eqnarray*}
This completes the proof.
\end{Proof}

\vspace{3mm}
\noindent{\bf Proof of Theorem~\ref{thm1}.}
Let $G$ be an $F$-free  graph  of   order  $n$ with the maximum $\alpha$-index.

 If $G$ is connected, then the result follows directly from Theorem~\ref{c-thm2}.
Next we suppose that $G$ is not connected. Since $K_{n-1}\nabla \overline{K}_{n-k+1}$ is $F$-free,
$$\rho_\alpha(G)\geq\rho_\alpha(K_{n-1}\nabla \overline{K}_{n-k+1})\geq \alpha n+\frac{2k-3-(2k-1)\alpha}{2\alpha}.$$ Let $G_1$ be a component of $G$  such that $\rho_\alpha(G_1)=\rho_\alpha(G)$. Set $n_1=|V(G_1)|$.
Then
\begin{eqnarray*}
% \nonumber to remove numbering (before each equation)
  n_1-1 &\geq&  \rho_\alpha(G_1)=\rho_\alpha(G)\geq\alpha n+\frac{2k-3-(2k-1)\alpha}{2\alpha}\\
  &\geq& \frac{4(\sum_{i=1}^k d_i+k-2)(\sum_{i=1}^k d_i+3k-5)}{\alpha^2}+\frac{2k-3-(2k-1)\alpha}{2\alpha},
\end{eqnarray*}
which implies that
\begin{eqnarray*}
% \nonumber to remove numbering (before each equation)
  n_1 &\geq &\frac{4(\sum_{i=1}^k d_i+k-2)(\sum_{i=1}^k d_i+3k-5)}{\alpha^2}+\frac{(2k-3)(1-\alpha)}{2\alpha} \\
   &>& \frac{4(\sum_{i=1}^k d_i+k-2)(\sum_{i=1}^k d_i+3k-5)}{\alpha^2}
\end{eqnarray*}
By Theorem~\ref{c-thm2} again, $\rho_\alpha(G_1)$ is no more than the largest root of
$$x^2-\Big(\alpha n_1+k+d_k-3\Big)x+
    \Big(\alpha(n_1-k+1)+k-2\Big)\Big(\alpha(k-1)+d_k-1\Big)-(1-\alpha)^2(k-1)(n_1-k+1)=0.$$
    Hence $\rho_\alpha(G_1)$ is less than the largest root of
$$x^2-\Big(\alpha n+k+d_k-3\Big)x+
    \Big(\alpha(n-k+1)+k-2\Big)\Big(\alpha(k-1)+d_k-1\Big)-(1-\alpha)^2(k-1)(n-k+1)=0.$$
 This completes the proof.
\QEDB
\vspace{3mm}

Let $F_{n,k}=K_{k-1}\nabla ( p K_2\cup qK_1)$, where $n-(k-1)=2p+q$ and $0\leq q<2$.

\begin{Corollary}\label{Cor}
Let $F=\cup_{i=1}^k S_{d_i}$ be a star forest with $k\geq2$ and $d_1\geq\cdots \geq d_k=2$.
If $G$ is an $F$-free graph of  order $n\geq \frac{4(\sum_{i=1}^k d_i+k-2)(\sum_{i=1}^k d_i+3k-5)}{\alpha^3}$,
 then $$\rho_\alpha(G)\leq \rho_\alpha(F_{n,k})$$ with equality if and only if $G=F_{n,k}$.
\end{Corollary}
\begin{proof} Let $G$ be a graph having the maximum $\alpha$-index among all $F$-free graphs of order $n$. It suffices to show that $G=F_{n,k}$.
If $G$ is connected, then by the proof of Theorem~\ref{c-thm2},  we have $G=K_{k-1}\nabla H,$
where $H$ is a graph of order $n-k+1$ with  $\Delta(H)\leq 1$. So $H$ is the union of some edges and isolated vertices. Hence $G=F_{n,k}$.
If $G$ is not connected, then by the similar proof of Theorem~\ref{thm1},  there is a component $G_1$ of $G$ such that $|V(G_1)|\geq  \frac{4(\sum_{i=1}^k d_i+k-2)(\sum_{i=1}^k d_i+3k-5)}{\alpha^2}$ and $\rho(G)=\rho(G_1)$.
By above case,
$$\rho_\alpha(G)=\rho_\alpha(G_1) =\rho_\alpha(F_{n_1,k})<\rho_\alpha(F_{n,k}).$$
 Hence the result follows.
\end{proof}

\vspace{3mm}

Let $\alpha=\frac{1}{2}$. By Theorem~\ref{thm1} and Corollary~\ref{Cor}, we have the following corollary.

\begin{Corollary}\label{Cor2}
Let $F=\cup_{i=1}^k S_{d_i}$ be a star forest with $k\geq2$ and $d_1\geq\cdots \geq d_k\geq2$ and $G$ be an $F$-free graph of order $n\geq 32(\sum_{i=1}^k d_i+k-2)(\sum_{i=1}^k d_i+3k-5)$.\\
(i) If $d_k=2$, then $$q(G)\leq q(F_{n,k})$$ with equality if and only if $G=F_{n,k}$.\\
(ii) If $d_k\geq 3$, then
 $$q(G)\leq \frac{n+2k+2d_k-6+\sqrt{(n+2k-2d_k-2)^2+8(k-1)(d_k-k+1)}}{2}$$ with equality if and only if $G=K_{k-1}\nabla H$, where $H$ is a $(d_k-1)$-regular graph of order $n-k+1$.
\end{Corollary}

\end{document}